\documentclass[10pt]{amsart}
\usepackage[cp1251]{inputenc}
\usepackage[english,russian]{babel}
\usepackage{amsmath}
\usepackage{amssymb}
\usepackage{amsfonts}

\newtheorem{theorem}{Theorem}

\newtheorem{corollary}{Corollary}
\newtheorem{proposition}{Proposition}

\begin{document}
\renewcommand{\refname}{References}
\renewcommand{\proofname}{Proof.}
\thispagestyle{empty}

\title[]{A note on regular subgroups of the automorphism group of the
linear Hadamard code}
\author{{I.Yu.Mogilnykh}}%
\address{Ivan Yurevich Mogilnykh
\newline\hphantom{i} Tomsk State University, Regional Scientific and Educational Mathematical Center, Tomsk, Russia
\newline\hphantom{ii} Sobolev Institute of Mathematics,
Novosibirsk, Russia}
\email{ivmog84@gmail.com}

\thanks{\sc Mogilnykh, I. Yu.,
A note on regular subgroups of the automorphism group of the
linear Hadamard code}
\thanks{\copyright \ 2018 Mogilnykh I.Yu.}
\thanks{\rm The work was supported by the Ministry of Education and Science of Russia (state assignment No. 1.12877.2018/12.1}

\vspace{1cm} \maketitle {\small
\begin{quote}
\noindent{\sc Abstract.} We consider the regular subgroups of the
automorphism group of the linear Hadamard code. These subgroups
correspond to the regular subgroups of $GA(r,2)$, w.t.r action on
the vectors of $F_2^{r}$, where $n=2^r-1 $ is the length of the
Hamadard code. We show that the dihedral group $D_{2^{r-1}}$ is a
regular subgroup of $GA(r,2)$ only when $r=3$. Following the
approach of \cite{M} we study the regular subgroups of the Hamming
code obtained from the regular subgroups of the automorphism group
of the Hadamard code of length 15.

\medskip

 \end{quote}
}

\section{Introduction}

Let $x$ be a binary vector of the $n$-dimensional vector space
$F_2^n$, $\pi$ be a permutation of the coordinate positions of
$x$. Consider the transformation $(x,\pi)$ acting on a binary
vector $y$ by the following rule:
 $$(x,\pi)(y)=x+\pi(y),$$
where $\pi(y)=(y_{\pi^(1)},\ldots,y_{\pi(n)})$.
 The composition of two
automorphisms $(x,\pi)$, $(y,\pi')$ is defined as follows
$$(x,\pi)\cdot(y,\pi')=(y+\pi(x),\pi'\circ\pi),$$
where $\circ$ is the composition of permutations $\pi$ and $\pi'$.

The automorphism group of the Hamming space $F_2^n$ is defined to
be $\mathrm{Aut}(F_2^n)=$  $\{(x,\pi): x \in C, \pi\in S_n, \,\,
x+\pi(F_2^n)=F_2^n\}$ with the operation composition, here $S_n$
denotes the group of symmetries of order $n$.

 The {\it automorphism group} $\mathrm{Aut}(C)$ of a code $C$ is
 the setwise stabilizer of $C$ in $Aut(F_2^n)$. In sequel for the sake of simplicity we require the all-zero vector, which we denote by $0^n$
  to be always in the code. Then we have the following representation
$$\mathrm{Aut}(C)=\{(x,\pi): x \in C, \pi\in S_n, \,\,
x+\pi(C)=C\}.$$ The {\it symmetry group} (also known as the
permutation automorphism group) of $C$ is defined as
$$\mathrm{Sym}(C)=\{\pi\in S_n: \pi(C)=C\}.$$

A code $C$ is called {\it transitive} if there is a subgroup $H$
of $\mathrm{Aut}(C)$ acting transitively on the codewords of $C$.
If the order of $H$ coincides with the size of $C$, then $H$
acting on $C$ is called a {\it regular group} \cite{PR} (sometimes
called sharply-transitive) and the code $C$ is called {\it
propelinear}.

  Propelinear codes provide a general
view on linear and additive codes, many of which are optimal. The
concept is specially important in cases where there are many
nonisomorphic codes with the same parameters, separting the codes
that are "close"\,\, to linear. In particular, among propelinear
codes there are $Z_2Z_4$-linear codes that could be defined via
Gray map. Generally, Hadamard codes are codes that could be
obtained from a Hadamard matrix of order $n$. Some researchers
consider Hadamard codes of length $n$, augmented by all-ones
vector, others study their shortenings of length $n-1$.
$Z_2Z_4$-linear perfect codes were classified in \cite{BR}, while
$Z_2Z_4$-linear Hadamard codes were classified in \cite{BPR},
\cite{K1}, along with the description of their automorphism groups
in  \cite{KV}. In work \cite{RR} $Z_2Z_4Q_8$-Hadamard codes are
discussed. Another point of study is finding a proper
generalization of the Gray map, and its further implementation for
construction of codes, see \cite{K2} for a study on
$Z_{2^k}$-linear Hadamard codes.

In below by the {\it Hadamard code} ${\mathcal A_n}$ we mean the
linear Hadamard code, i.e. of length $n=2^r-1$, dimension $r$ and
minimum distance $(n+1)/2$. The code is dual to the Hamming code,
which we denote by ${\mathcal H_n}$, so their symmetry groups
coincide and ${\mathcal A_n}$ is unique up to a permutation of
coordinate positions.

In Section 2 of the current paper we give auxiliary statements. In
particular, we show that the regular subgroups of $Aut({\mathcal
A_n})$ correspond to those of $GA(log(n+1),2)$ and give a bound on
the order of an element of a regular subgroup of $GA(r,2)$.

 There are few
references on regular subgroups of the affine group from strictly
algebraic point of view. A regular subgroup of $GA(r,q)$ without
nontrivial transla-tions was constructed in \cite{H}. In works
\cite{C}, \cite{Ch} it was shown that the abelian regular
subgroups of $GA(r,q)$ correspond to certain algebraic structures
on the vector space $F_q^r$. In work \cite{Ch} the following
example of abelian regular subgroup of $GA(r,q)$ was mentioned:
the group is the centralizer of the Jordan block of size $r+1$ in
the group of upper triangular matrices \cite{Ch}.



One of  main problems, arising in the theory of propelinear codes
is a construction of codes with regular subgroups in their
automorphism group that are abelian or  "close"\,\, to them in a
sense, such as for example $Z_4^l$, cyclic or dihedral groups. The
same question could be asked for the regular subgroups of the
affine group. In Section 3, we see that the dihedral group is a
regular subgroup of the affine group if and only if $r=3$, with
the nontrivial case of the proof being when $r$ is $4$, when a
there is a dihedral subgroup of the affine group, which is not
regular.

The Hamming code ${\mathcal H}_n$ is known to have the largest
order of the automorphism group in the class of perfect binary
codes of any fixed length \cite{ST} and it would be natural to
suggest that it has the maximum number of regular subgroups of its
automorphism group among propelinear perfect codes. However, the
fact that
 $$|Aut({\mathcal H_n})|=|GL(log(n+1),2)|2^{n-log(n+1)}$$
  makes attempts of even partial
classification of regular subgroups impossible for ordinary
calculational machinery starting with the smallest nontrivial
length $n=15$.

Regular subgroups of the Hamming code could be constructed from
the regular subgroups of the automorphism group of its subcodes
whose automorphism groups are embedded into that of the Hamming
code in a certain way. In work \cite{M}, this idea was implemented
for the Nordstrom-Robinson code in case of extended length $n=16$.
In Section 4 we embed regular subgroups of the Hadamard code into
those of the Hamming code of length $15$.

\section{Preliminaries}

We begin with the following two well-known facts, e.g. see
\cite{MS}.

\begin{proposition}\label{P2} Let $C$ be a linear code of length $n$. Then

$$\mathrm{Aut}(C)=F_2^n\leftthreetimes Sym(C)=\{(x,\pi):x\in C, \pi \in \mathrm{Sym}(C)\}.$$
\end{proposition}

The Hadamard code is known to be the dual code of the Hamming code
of length $n=2^r-1$, which implies that their symmetry groups
coincide and is isomorphic to the general linear group of $F_2^r$.

\begin{proposition}\label{P1}
Let $\mathcal A_n$ and $\mathcal H_n$ be the Hadamard and the
Hamming codes of length $n=2^r-1$. Then

 $$Sym(\mathcal H_n)=Sym(\mathcal A_n)\cong GL(r,2).$$

\end{proposition}

As far as the automorphism groups are concerned, the following
fact holds.

\begin{proposition}
Let $\mathcal A_n$ be the Hadamard code of length $n=2^r-1$. Then
$\mathrm{Aut}(\mathcal A_n)\cong GA(r,2)$ and the action of
$\mathrm{Aut}(\mathcal A_n)$ on the codewords of $\mathcal A_n$ is
equivalent to the natural action of $GA(r,2)$ on the vectors of
$F_2^r$. In particular, the regular subgroups of
$\mathrm{Aut}(\mathcal A_n)$ correspond to the regular subgroups
of $GA(r,2)$.
\end{proposition}
\begin{proof}
We use a well-known representation of the Hadamard code, see e.g.
\cite{MS}. For a vector $a\in F_2^r$ consider the vector $c_a$ of
values of the function $\sum_{i=1,\ldots,n} x_i a_i$ of variable
$x$ from $F_2^r\setminus 0^r$ to $F_2$. It is easy to see that the
code $\mathcal A_n=\{c_a:a \in F_2^r\}$ is linear of length
$n=2^r-1$, dimension $r$ and minimum distance $(n+1)/2$, i.e.
$\mathcal A_n$ is the Hadamard code. By Propositions \ref{P2} and
\ref{P1} any automorphism of $\mathrm{Aut}(\mathcal A_n)$ is
$(c_a,\pi_A)$ for a vector $a\in F_2^r$ and $A\in GL(r,2)$,
therefore the mapping $(c_a,\pi_A)\rightarrow (a,A)$ is an
isomorphism from $\mathrm{Aut}(\mathcal A_n)$ to $GA(r,2)$.

\end{proof}

In \cite{B} the maximal orders of elements of $GL(r,q)$ were
described. In particular, the following was shown:
\begin{proposition}\label{orderGLupperbound}
 The maximum of orders of elements of $GL(r,2)$ of type $2^l$
is $2^{1+\lfloor log_2(r-1)\rfloor}$.
\end{proposition}

 This implies
that the order of the element of regular subgroup of $GA(r,2)$
does not exceed $$2^{2+\lfloor log_2(r-1)\rfloor}.$$ This fact
solely implies the nonexistence of regular dihedral subgroups of
$GA(r,2)$ for $r\geq 6$. In fact we can tighten the bound to
$r\geq 5$.

\begin{proposition}\label{Prop1}

    1.  Let $A$ be an element of $GL(r,2)$ of order $2^l$. Then $(I+A)^r=0$.
  \noindent  2.  The order of an element of a regular subgroup of
    $GA(r,2)$ is not greater then $2^{{\lfloor log_2 r \rfloor}+1}$.

 \end{proposition}
\begin{proof}
1. The order of $A$ is $2^l$, then $(\lambda+1)^{2^l}=0$ for any
eigenvalue $\lambda$ of $A$, which implies that all eigenvalues of
$A$ are 1's and w.r.g. $A$ is in the Jordan form with the Jordan
blocks $J_1,\ldots, J_s$ corresponding to 1. It is easy to see
that the polynomial $(I+J_i)^r$ of the Jordan cell $J_i$ is zero
for $J_i$ of size not greater then $r$.

 2. Suppose that $(a,A)$ is an element of a regular subgroup of $GA(r,2)$ of order greater then $2^{{\lfloor log_2 r
 \rfloor}+1}$.

 We see that
 $(a,A)^i= (\sum_{j=0,\ldots,i-1} A^{j}a,A^i)$, which combined
 with the fact that binomials $(^{2^s-1}_{i})=1$ in $F_2$ for any $i:0\leq i\leq 2^s-1$, implies that:
$$(a,A)^{2^s}=((I+A)^{2^s-1}a,A^{2^s}).$$
In particular, using that $(I+A)^r=0$, we have that
$$(a,A)^{2^{\lfloor log_2 r \rfloor+1}}=((I+A)^{2^{\lfloor log_2 r
\rfloor+1}-1}a,A^{2^{\lfloor log_2 r
\rfloor+1}})=(0^r,A^{2^{\lfloor log_2 r \rfloor+1}}).$$ Therefore
distinct elements $(0^r,I)$ and $(0^r,A^{2^{\lfloor log_2 r
\rfloor}+1})$ of a regular subgroup both preserve $0^r$, a
contradiction.

\end{proof}

We finish the section by noting that a version of the direct
product construction works for regular subgroups.

\begin{proposition}\label{DPC}
Let $G$ and $G'$ be regular subgroups of $GA(r,2)$ and $GA(r',2)$.
Then there is a regular subgroup of $GA(r+r',2)$ isomorphic to
$G\times G'$.
\end{proposition}
\begin{proof}
 Given elements $\alpha=(a,A)$
of $GA(r,2)$ and $\beta=(b,B)$ of $GA(r',2)$, define $\alpha\cdot\beta$ to be $((a|b),\left(%
\begin{array}{cc}
  A & 0^{rr'} \\
  0^{r'r} & B \\
\end{array}%
\right))$, where $(a|b)$ is the concatenation of vectors $a$ and
$b$. Obviously, the elements $\{\alpha\cdot \beta: \alpha\in G,
\beta \in G'\}$ form a regular subgroup of $GA(r+r',2)$,
isomorphic to $G\times G'$.

\end{proof}

\section{Dihedral regular subgroups of $GA(r,2)$}

The cyclic group $Z_{2^r}$ is not a regular subgroup of $GA(r,2)$
for any $r$, as we see from the bound in Proposition \ref{Prop1}.
Therefore, we address the question of being a regular subgroup of
the affine group to other groups, that are "close"\,\,to cyclic.
The dihedral group, which we denote by $D_{n}$, is the group
composed by all $2n$ symmetries of the $n$-sided polygon. It is
well-known that any group, generated by an element $\alpha$ of
order $n$ and an involution $\beta$ satisfying
$\beta\alpha\beta=\alpha^{-1}$ is isomorphic to $D_n$.

\begin{theorem}
$D_{2^{r-1}}$ is a regular subgroup of $GA(r,2)$ if and only if
$r=3$.
\end{theorem}
\begin{proof}
 Consider the subgroup $G$ of $GA(3,2)$ generated by ${(a,A),
(b,I)}$, where
$a=(101)^T$, $b=(011)^T$, $A=\left(%
\begin{array}{ccc}
  1 & 0 & 1 \\
  0 & 1 & 0 \\
  0 & 0 & 1 \\
\end{array}%
\right)$. The orbit of $(000)^T$ under the action of the subgroup
generated by $(a,A)$ consists of vectors
$a=(101)^T,a+Aa=(100)^T,Aa=(001)^T,(000)^T$. Since the orbit is a
subspace that does not contain the vector $b$, $G$ acts
transitively on the elements of $F_2^3$. Moreover,
$(b,I)(a,A)(b,I)=(Ab+a+b,A)=((001)^T,A)=(a,A)^{-1}$, so $G$ is
$D_{4}$ and it is regular.

Suppose there is a regular subgroup of $GA(4,2)$, generated by an
element $(a,A)$ of order 8 and an element $(b,B)$ of order 2,
satisfying relation \begin{equation}\label{P0}
(b,B)(a,A)(b,B)=(a,A)^{-1}.\end{equation} Note that the order $A$
is 4 by Proposition \ref{orderGLupperbound}.

Since $(b,B)^2=(0^4,I)$, we have that
\begin{equation}\label{T1}
b=Bb.
\end{equation}

Taking into account relation $b=Bb$ we have the following:

$$(a,A)^{-1}=(a,A)^{7}=(\sum_{j=0,\ldots,6}A^{j}a,A^3)=(A^{3}a,A^3)=$$ $$(b,B)(a,A)(b,B)=(Ba+BA^3b+b,BAB),$$
therefore using (\ref{T1}) and $BAB=A^3$, we obtain:

\begin{equation}\label{P00}
Ba=A^3b+A^3a+b.
\end{equation}

The matrix $A$ is similar to the Jordan block of size 4 with the
eigenvalue 1. Since $(a,A)^4=((I+A)^3a,I)$, the vector $(I+A)^3a$
is nonzero and moreover is the unique eigenvector of $A$. The
Jordan chain (the basis for which $A$ is the Jordan block)
containing $(I+A)^3a$ are vectors $a,(I+A)a, (I+A)^2a,(I+A)^3a$,
which implies that $a,Aa, A^2a, A^3a$ is a basis of $F_2^4$, so

\begin{equation} \label{T2}
b=c_0a+c_1Aa+c_2A^2a+c_3A^3a,
\end{equation}
for some $c_i$ in $F_2$, $i \in \{0,\ldots,3\}$.

Putting the expression (\ref{T2}) for $b$ into the equality
(\ref{P00}), we obtain the following expression for $Ba$:
\begin{equation}\label{Ba}
Ba=(c_0+c_1)a+(c_1+c_2)Aa+(c_2+c_3)A^2a+(c_0+c_3+1)A^3a.
\end{equation}
Putting the expression (\ref{T2}) for $b$  into (\ref{T1}) and
using equality $BAB=A^3$, we obtain:

$$c_0Ba+c_1(A^3)Ba+c_2(A^2)Ba+c_3(A)Ba+c_0a+c_1Aa+c_2A^2a+c_3A^3a=0^4.$$
Substituting the expression (\ref{Ba}) for $Ba$ in the previous
equality, we obtain that:

$$(c_{0}c_3+c_{0}c_1+c_{1}c_2+c_{2}c_3+c_2+c_3)(a+Aa+A^2a+A^3a)=0^4.$$
Finally, we see that the only binary vectors  $(c_0,c_1,c_2,c_3)$
satisfying
 $$c_{0}c_3+c_{0}c_1+c_{1}c_2+c_{2}c_3+c_2+c_1=0$$
are exactly
$$(0000),(1000),(1100),(1110),(1111),(0111),(0011),(0001),$$
that are, in turn, exactly coefficients of linear combinations
expressing elements $\sum_{j=0,\ldots,i} A^ja$, $0 \leq i\leq 7$
in the basis $a, Aa, A^2a, A^3a$. Therefore, elements
$(b,B)=(\sum_{j=0,\ldots,i} A^ja, B)$ and
$(a,A)^{i+1}=(\sum_{j=0,\ldots,i} A^ja, A^{i+1})$ are distinct
elements of the dihedral subgroup, for some $i$, sending $0^4$ to
$b$. We conclude that the considered group is not regular.

Suppose there is a regular subgroup $D_{2^{r-1}}$ of $GA(r,2)$,
$r\geq 5$. Then there is an element in $D_{2^{r-1}}$ of order
$2^{r-1}$ which is impossible for $r\geq 5$, because the order of
an element in a regular subgroup of $GA(r,2)$ does not exceed
$2^{\lfloor log r\rfloor+1}$ by Proposition \ref{Prop1}.

\end{proof}

{\bf Remark 1.} The subgroup of $GA(4,2)$ generated by
$((0001)^T,A)$ and $((0000)^T, B)$, where
$A=\left(%
\begin{array}{cccc}
  1 & 1 & 0 & 0\\
  0 & 1 & 1 & 0\\
  0 & 0 & 1 & 1\\
  0 & 0 & 0 & 1\\
\end{array}%
\right)$, $B=\left(%
\begin{array}{cccc}
  1 & 1 & 1 & 1\\
  0 & 1 & 0 & 1\\
  0 & 0 & 1 & 1\\
  0 & 0 & 0 & 1\\
\end{array}%
\right)$ is a irregular subgroup of $GA(4,2)$ isomorphic to $D_8$.

{\bf Remark 2.} Consider elements  $((0001)^T,A)$ and
$((0100)^T,B)$, where $A$ is the same as in Remark 1,
$B=\left(%
\begin{array}{cccc}
  1 & 0 & 0 & 1\\
  0 & 1 & 0 & 0\\
  0 & 0 & 1 & 0\\
  0 & 0 & 0 & 1\\
\end{array}%
\right)$. It is easy to see that the element $((0100)^T,B)$ is an
involution, $((0001)^T,A)$ is of order $8$, they commute and
moreover they generate a regular subgroup isomorphic to $Z_2Z_8$.
Actually, the group is isomorphic to the abelian regular subgroup,
arising from the centralizer of the Jordan block of size 5 in the
group of upper triangular $5\times 5$ matrices, described in
\cite{Ch}.

\section{Embedding to regular subgroups of the automorphism group of the Hamming code of length 15}


Regular subgroups of the Hamming code could be constructed from
regular subgroups of its subcodes whose automorphism groups are
embedded into that of the Hamming code in a certain way. In work
\cite{M}, narrow-sense embeddings of the regular subgroups of the
automorphism group of Nordstrom-Robinson code to those of the
Hamming code were considered in case of extended length $n=16$.
Here we apply the idea to embed regular subgroups of the
automorphism group of the Hadamard code to regular subgroups of
the automorphism group the Hamming code.

Given a subgroup $G$ of $Aut(F_2^n)$, denote by $\Pi_G$ the
subgroup of $S_n$ whose elements are $\{\pi: (x,\pi)\in G\}$.
 We say that a group $H, H\leq Aut(F_2^n)$ is {\it narrow-sense
embedded} \cite{M} in a
 subgroup $G, G\leq Aut(F_2^n)$, if $H\leq G$ and $\Pi_G=\Pi_H$.

It is well-known that the Hadamard code ${\mathcal A}_{15}$ and
the punctured Nordstrom-Robinson, which we denote by ${\mathcal
N}$, are subcodes of the code ${\mathcal H}_{15}$ \cite{MS},
\cite{SZ}. Obviously, $\Pi_{Aut(C)}=Sym(C)$ if $C$ is linear. So,
$Aut({\mathcal A}_{15})$
 is narrow-sense embedded into that of the
Hamming code ${\mathcal H}_{15}$, because their symmetry groups
coincide (see Proposition \ref{P1}). The linear span of the
punctured Nordstrom-Robinson code is the Hamming code \cite{SZ},
thus its automorphism group is embedded in that of the Hamming
code. Moreover, the inclusion is in narrow sense. We recall a
description of symmetry group of ${\mathcal N}$ from \cite{Ber}.

\begin{proposition}\label{PB}
$Sym(\mathcal N)\cong A_7< Sym({\mathcal H}_{15})\cong GL(4,2)
\cong A_8$.
\end{proposition}

\begin{corollary}
$Aut(\mathcal N)$ is narrow-sense embedded in $Aut({\mathcal
H}_{15})$.
\end{corollary}
\begin{proof}
The punctured  Nordstrom-Robinson code is propelinear, then it is
not hard to see that $|\Pi_{Aut({\mathcal N})}|=|Sym({\mathcal
N})||\mathcal N|/|Ker(\mathcal N)|$, where $Ker({\mathcal N})=\{x
\in {\mathcal N}:x+{\mathcal N}={\mathcal N }\}$ see e.g.
\cite{BMRS}, Proposition 4.3. Then, since $Ker({\mathcal N})$ is
${\mathcal A}_{15}$ augmented by all-ones vector, see \cite{SZ}
and the size of ${\mathcal N}$ is $2^8$, we see that
\begin{equation}\label{EC}|\Pi_{Aut({\mathcal N})}|=8|Sym(N)|=|Sym({\mathcal
H}_{15})|.\end{equation}

Let $\pi$ be an element of $\Pi_{Aut({\mathcal N})}$, in other
words, $x+\pi({\mathcal N})={\mathcal N}$. The linear span of
${\mathcal N}$ is ${\mathcal H}_{15}$  \cite{SZ}, therefore $\pi$
is a symmetry of ${\mathcal H}_{15}$. Taking into account the
equality (\ref{EC}), we obtain that $\Pi_{Aut({\mathcal
N})}=Sym({\mathcal H}_{15})=\Pi_{Aut({\mathcal H}_{15})}$.

\end{proof}

First of all, the regular subgroups of the automorphism group of
$Aut({\mathcal A}_{15})$ (regular subgroups of $GA(4,2)$) were
classified. The results below were obtained using PC.

\begin{theorem} There are 39 conjugacy classes of regular
subgroups of $Aut({\mathcal A}_{15})$, that fall into 11
isomorphism classes. \end{theorem}

{\bf Remark 3.} Four of 11 isomorphism classes of regular
subgroups are abelian and are isomorphic to groups $Z_2^4$,
$Z_2Z_8$, $Z_2^2Z_4$ and $Z_4^2$. A group isomorphic to $Z_2Z_8$
is given in Remark 2. It is not hard to see that there is a
regular subgroup of $GA(2,2)$, isomorphic to $Z_4$. Then the
regular subgroups isomorphic to $Z_2^2Z_4$ and $Z_4^2$ could be
constructed
 using direct product construction (see Proposition \ref{DPC}).

The narrow-sense embeddings into regular subgroups of the
automorphism group of the Hamming code were found. The bound on
the number of isomorphism classes we obtain by comparing the
orders of the centralizers of elements.

\begin{theorem}
The regular subgroups of $Aut({\mathcal A}_{15})$ are narrow-sense
embedded in at least 1207 conjugacy classes of regular subgroups
of $Aut({\mathcal H}_{15})$, which fall into at least 48
isomorphism classes.
\end{theorem}

The result is somewhat disappointing, as embeddings of
Nordstrom-Robinson code in Hamming code gave significantly better
bound for isomorphism classes.

\begin{theorem} \cite{M}
There are 73 conjugacy classes of regular subgroups of
$Aut({\mathcal N})$ that fall into 45 isomorphism classes. The
regular subgroups of $Aut({\mathcal N})$ are  narrow-sense
embedded in exactly 605 conjugacy classes of regular subgroups of
$Aut({\mathcal H}_{15})$, which fall into at least 219 isomorphism
classes.
\end{theorem}

One might suggest a tighter interconnection of regular subgroups
of $Aut({\mathcal A}_{15})$ and that of $Aut({\mathcal N})$.
However, despite that ${\mathcal A}_{15}\subset {\mathcal N}$,
$Aut({\mathcal A}_{15})$ is not embedded in that of $Aut(N)$ in
narrow-sense, which in turn, follows, for example, from a proper
containment of $Sym({\mathcal N})$ in $Sym({\mathcal A}_{15})$,
see Proposition \ref{PB}. Moreover, only $6$ of $39$ conjugacy
classes of regular subgroups of $Aut({\mathcal A}_{15})$ are
subgroups $Aut({\mathcal N})$.

 The author is grateful to Fedor Dudkin and
Alexey Staroletov for stimulating discussions and pointing out the
work \cite{B}.


%

\end{document}